\documentclass{amsproc}
\usepackage{hyperref}
\usepackage{color}
\usepackage{graphicx}
\usepackage{graphics,amsmath,amsfonts,amssymb} 
\usepackage{epsfig,verbatim,multicol}

\newtheorem{teor}{Theorem}[section]
\newtheorem{cor}[teor]{Corollary}
\newtheorem{lemma}[teor]{Lemma}

\newtheorem{prop}[teor]{Proposition}
\theoremstyle{definition}

\theoremstyle{remark}

\newtheorem*{Ejem}{Example}

\newtheorem{obs}{Remark}[section]
\numberwithin{equation}{section}

\begin{document}

	\title[Algebraic and qualitative remarks about the family...]{Algebraic and qualitative remarks about the family $yy'= (\alpha x^{m+k-1} + \beta x^{m-k-1})y + \gamma x^{2m-2k-1}$}
	
	\author[P. Acosta-Hum\'anez]{Primitivo B, Acosta-Hum\'anez}
	\address[P. Acosta-Hum\'anez]{Instituto Superior de Formaci\'on Docente Salom\'e Ure\~na - ISFODOSU, Recinto Emilio Prud'Homme, Santiago de los Caballeros - 
		Dominican Republic \& Universidad Sim\'on Bol\'{\i}var, Barranquilla - Colombia}
	\email{primitivo.acosta-humanez@isfodosu.edu.do}
	
	\author[A. Reyes-Linero]{Alberto Reyes-Linero}
	\address[A. Reyes-Linero]{Facultad de ciencias ba\'asicas Universidad del Atl\'antico, Barranquilla - Colombia}
	\email{areyeslinero@mail.uniatlantico.edu.co}
	
	\author[J. Rodriguez-Contreras]{Jorge Rodr\'{i}guez-Contreras}
	\address[J. Rodriguez-Contreras]{Departamento de de matem\'atica y estad\'istica Universidad del Norte \& Facultad de ciencias b\'asicas Universidad del Atl\'antico, Barranquilla - Colombia}
	\email{jrodri@uninorte.edu.co - jorgelrodriguezc@mail.uniatlantico.edu.co}
	
\begin{abstract}
The aim of this paper is the analysis, from algebraic point of view and singularities studies, of the 5-parametric family of differential equations 
		\begin{equation*}\label{folpz}
		yy'=(\alpha x^{m+k-1}+\beta x^{m-k-1})y+\gamma x^{2m-2k-1}, \quad y'=\frac{dy}{dx}
		\end{equation*}
		where  $a,b,c\in \mathbb{C}$, $m,k\in \mathbb{Z}$ and
		$$\alpha=a(2m+k) \quad \beta=b(2m-k), \quad \gamma=-(a^2mx^{4k}+cx^{2k}+b^2m).$$
This family is very important because include Van Der Pol equation. Moreover, this family seems to appear as exercise in the celebrated book of Polyanin and Zaitsev. Unfortunately, the exercise presented a typo which does not allow to solve correctly it. We present the corrected exercise, which corresponds to the title of this paper. We solve the exercise and afterwards we make algebraic  and of singularities studies to this family of differential equations. To illustrate the qualitative and algebraic techniques we present an example of a biparametric quadratic Polyanin-Zaitsev vector field.\medskip

		\noindent\footnotesize{\textbf{Keywords and Phrases}. \textit{Critical points, integrability, Gegenbauer equation, Legendre equation, Li\'enard equation.}}\medskip

		\noindent\footnotesize{\textbf{MSC 2010}. Primary 12H05; Secondary 34C99}
	\end{abstract}
		\maketitle
	
	\section*{Introduction}
	Dynamical systems is a topic of interest for a large number of theoretical physicist and mathematicians due to the seminal works of H. Poincar\'e. It is well known that any dynamical system is a system which evolves in the time. H. Poincar\'e introduced the qualitative approach to study dynamical systems, which has been useful to study theoretical aspects and applications to biology, chemistry, physics, among others, see \cite{Jg,GHW,Tk,NAY,Ns,PK,VDP}.\\
	
	On another hand, E. Picard and E. Vessiot introduced an algebraic approach to study linear differential equations based on the Galois theory for polynomials, see \cite{A,Ap,AJp,AJp2,JJM,VDPS}. Combination of dynamical systems with differential Galois theory is a recent topic which started with the works of J.J Morales-Ruiz (see \cite{JJM} and references therein) and with the works of J.-A. Weil (see \cite{JAW}). Further works about applications of differential Galois theory include \cite{Ac3,AB,Amw}.\\
	
	The \emph{Handbook of Exact Solutions of Ordinary Differential Equations}, see  \cite{PY}, is one important reference for scientists and engineers interested in solving explicitly ordinary differential equations. This book contains around 3,000 nonlinear ordinary differential equations with
	solutions, as well as exact, symbolic, and numerical methods for solving nonlinear equations. Nonlinear equations and systems with first-, second-, third-, fourth-, and higher-order are considered there.\\ 
	
	Inspired by a previous version of the paper \cite{Almp}, we analysed the Exercise 11 in \cite[\S 1.3.3]{PY}, which corresponds to a five parametric family of differential equations. We discovered a typo (also corrected by us), which was corrected in the final version of \cite{Almp} to study from Differential Galois Theory point of view the integrability of the dynamical system proposed in such excercise. \\
	
	We call as \emph{Polyanin-Zaitsev vector field} to the vector field associated to this system of differential equations that comes from the corrigendum of the Exercise 11 in \cite[\S 1.3.3]{PY}. Moreover, we study integrability aspects using differential Galois theory, following \cite{Almp,Ap} as well qualitative aspects due  to the foliation associated to Polyanin-Zaitsev vector field is a Li\'enard equation, which is closely related to a Van Der Pol equation.\\
	
	This paper presents the corrigendum and complete solution of the Polyanin-Zaitsev excercise mentioned above. Moreover, this paper also extends the results given in \cite{Almp} concerning the Polyanin-Zaitsev vector field. From algebraic point of view we give conditions over the parameters to have polynomial vector field, while from qualitative point of view we obtain the finite critical points for some particular cases and we describe their behavior. We recall that in this paper we do not present a complete qualitative analysis of Polyanin-Zaitsev vector field, because we do not study orbits and infinite critical points. However, the study of finite critical points is very important in the qualitative theory of dynamical systems. 
	
	The algebraic results of this paper concerning the transformations of polynomials and differential equations were obtained, but not published, during the seminar \emph{Algebraic Methods in Dynamical Systems} in 2013, which was developed by the first author and it was also included in the master thesis of the second author in 2014 (supervised by the first and third author).

	\section{Preliminaries}
	In this section we provide the necessary theoretical background to understand the rest of the paper.\\
	
	A \emph{planar polynomial system} of degree $n$ is given by
	
	\begin{equation}\label{camlien1}
	\begin{array}{lll}
	\dot{x}&=&P(x,y)\bigskip\\
	\dot{y}&=&Q(x,y),
	\end{array}
	\end{equation}
	
	being $P,Q\in\mathbb{C}[x,y]$ and $n=\max(\deg P, \deg Q)$. By $ X:=(P,Q)$ we denote the \emph{polynomial vector field} associated to the system \eqref{camlien1}. The planar polynomial vector field $X$ can be also writen in the form
	$$X=P(x,y)\frac{\partial }{\partial x}+Q(x,y)\frac{\partial }{\partial y}. \hspace{1cm} \label{camvet1}$$\medskip

	A \emph{foliation} of a polynomial vector field of the form \eqref{camlien1} is given by $$\frac{dy}{dx}=\frac{Q(x,y)}{P(x,y)}.$$
	
	Following \cite{PK}, we present the following theorem, which allow us the characterization of the critical points.
	\begin{teor}\label{teorest}
		Let $X$, $Y$ be analytic functions with polynomial part containing terms of degree greater than 1. Consider the planar differential system
		$$\begin{array}{lcl}\dot{x}&=&y+X(x,y)\\\dot{y}&=&Y(x,y)\end{array}$$
		being  the origin an isolated
		critical point. Assume  $$y=F(x)=a_2x^2+a_3x^3+\ldots$$ as the solutions of $y+X(x,y)=0$ near to $(0,0)$. Suppose $$f(x)=Y(x,F(x))=ax^{\alpha}(1+\ldots), \quad a\neq0,\quad \alpha\geq2$$ and also  $$\Phi(x)=\left(\displaystyle{\frac{\partial X}{\partial
				x }+\frac{\partial Y}{\partial
				y}}\right)|_{(x,F(x))}=bx^{\beta}(1+\ldots),\quad  b\neq0, \quad \beta\geq1.$$
		\\
		Then the following statements hold:
		\begin{enumerate}
			\item[a)] If $\alpha$ is even and $\alpha>2\beta+1$, then $(0,0)$
			is a saddle node.
			\\
			If $\alpha$ is even and $\alpha<2\beta+1$ \textrm{or} $\Phi(x)\equiv 0$,
			then the flow near of $(0,0)$ have two hyperbolic sectors.
			
			\item[b)] If $\alpha$ is
			odd and $a>0$, then $(0,0)$ is a saddle point.
			
			\item[c)] If $\alpha$ is odd and $a<0$, several cases can occur:\\
			$$\begin{array}{l}
			\bullet\,\, c_1\quad \alpha>2\beta+1, \quad\beta\hspace{0.5cm}\textrm{even}\\
			\bullet \,\, c_2 \quad\alpha=2\beta+1, \quad\beta\hspace{0.5cm}\textrm{even}, \quad
			b^2+4a(\beta+1)\geq0,
			\end{array}$$
			in this case for $b<0$ the critical point $(0,0)$ is a stable node, while for $b>0$ the critical point $(0,0)$ is unstable node.
			$$\begin{array}{l}
			\bullet\,\, c_3\quad \alpha>2\beta+1, \quad \beta \quad \textrm{odd}\\
			\bullet \,\, c_4 \quad\alpha=2\beta+1, \quad \beta \quad \textrm{odd}, \quad
			b^2+4a(\beta+1)\geq0,
			\end{array}$$
			in this case the flow  near to the critical point $(0,0)$ is topologically conformed by an elliptic sector joint with an hyperbolic sector.
			
			$$\begin{array}{l}
			\bullet\,\, c_5\quad \alpha=2\beta+1\quad\textrm{y}\quad b^2+4a(\beta+1)<0\\
			\bullet\,\, c_6 \quad\alpha<2\beta+1, \quad \textrm{{or}}, \quad
			\Phi(x)\equiv 0,
			\end{array}$$ in this case the critical point $(0,0)$ is focus or center.
		\end{enumerate}
	\end{teor}
	
	\section{Corrigendum to the problem}
	
	The original Exercise 11, section 1.3.3 of the book of Polyanin-Zaitsev (see \cite[\S 1.3.3.11]{PY}) was presented as follows:\\
	
	$$yy'=(a(2m+k)x^{2k}+b(2m-k))x^{m-k-1})y-(a^2mx^{4k}+cx^{2k}+b^2m)x^{2m-2k-1}.$$
	\emph{The transformation}  $z=x^k$, $y=x^m(t+ax^k+bx^{-k})$ \emph{leads to a Riccati equation with respect to} $z=z(t)$:
	$$(-mt^2+2mab-c)\frac{dz}{dt}=bk+tkz+akz^2. \hspace{1cm}\textbf{(1)}$$
	\emph{The substitution} $z=\frac{mt^2+c_0}{ak}\frac{w'_t}{w}$, \emph{where} $c_0=c-2abm$, \emph{reduces equation (1) to a second order linear equation}:
	$$(mt^2+c_0)^2w_{tt}''+(2m+k)t(mt^2+c_0)w_t'+abk^2w=0. \hspace{1cm}\textbf{(2)}$$
	
	\emph{The transformation} $\xi=\frac{t}{\sqrt{t^2+c_0/m}}$ , $u=(1-\xi^2)^{\mu/2}w$ \emph{where} $\mu=-\frac{m+k}{2m}$  \emph{bring equation} \textbf{(2)} \emph{to the Legendre equation} 2.1.2.226:
	
	$$(1-\xi^2)u_{\xi\xi}''-2\xi u_{\xi}'+[ \nu(\nu+1)-\mu^2(1-\xi^2)^{-1} ]u=0$$
	\emph{where} $\nu$ \emph{is a root of the quadratic equation}
	$\nu^2+\nu+\frac{m^2-k^2}{4m^2}-\frac{abk^2}{mc_0}=0.$\bigskip\\	
	
	A typo in this exercise does not allows its solving. The correction of the problem is presented in the following proposition:

	\begin{prop}
		Given the family of Li\'enard equations of the form:
		$$yy'=(a(2m+k)x^{m+k-1}+b(2m-k)x^{m-k-1})y-(a^2mx^{4k}+cx^{2k}+b^2m)x^{2m-2k-1},$$
		the change of variables $w=x^k$ and $y=x^m(t+ax^k+bx^{-k})$ allow to transform any equation of this family to a Riccati equation.
	\end{prop}
	
	\begin{proof}\bigskip
		The system of equations, associated with this Li\'enard equation is:
		\begin{displaymath}
		\begin{array}{ccl}
		\dot{x}&=&y \bigskip \\
		\dot{y}&=& (a(2m+k)x^{m+k-1}+b(2m-k)x^{m-k-1})y-(a^2mx^{4k}+cx^{2k}+b^2m)x^{2m-2k-1}.
		\end{array}
		\end{displaymath}
		
		Now, applying the transformation
		$$z=x^k, \hspace{1cm} y=az^{\frac{m+k}{k}}+tz^{\frac{m}{k}}+bz^{\frac{m-k}{k}}$$
		and differentiating we have that:
		$$dx=\frac{1}{k}z^{\frac{1-k}{k}}dz, \hspace{1cm} dy=z^{\frac{m}{k}}dt+(t\frac{m}{k}z^{\frac{m-k}{k}}+a\frac{m+k}{k}z^{\frac{m}{k}}+b\frac{m-k}{k}z^{\frac{m-2k}{k}})dz,$$
		
		then, the associated foliation has the form $ydy=(f(x)y-g(x))dx,$ being $f(x)$ and $g(x)$ as the Li\'enard equations. 	Now we compute each part of this equality, thus we obtain the left side as:
		\begin{displaymath}
		\begin{array}{ccl}
		kydy&=&k(az^{\frac{m+k}{k}}+tz^{\frac{m}{k}}+bz^{\frac{m-k}{k}})z^{\frac{m}{k}}dt+k(t^2\frac{m}{k}z^{\frac{2m-k}{k}}+at\frac{m+k}{k}z^{\frac{2m}{k}}\\
		& &bt\frac{m-k}{k}z^{\frac{2m-2k}{k}}+a^2\frac{m+k}{k}z^{\frac{2m+k}{k}}+at\frac{m}{k}z^{\frac{2m}{k}}+ab\frac{m-k}{k}z^{\frac{2m-k}{k}}\\ & &+tb\frac{m}{k}z^{\frac{2m-2k}{k}}+ab\frac{m+k}{k}z^{\frac{2m-k}{k}}+b^2\frac{m-k}{k}z^{\frac{2m-3k}{k}})dz\bigskip \\
		&=&k(akz^2+tkz^+bk)z^{\frac{2m-k}{k}}dt+(t^2mz^{\frac{2m-k}{k}}+at(m+k)z^{\frac{2m}{k}}+bt(m-k)z^{\frac{2m-2k}{k}}\\
		& &+a^2(m+k)z^{\frac{2m+k}{k}}+atmz^{\frac{2m}{k}}+ab(m-k)z^{\frac{2m-k}{k}}+tbmz^{\frac{2m-2k}{k}}\\
		& &+ab(m+k)z^{\frac{2m-k}{k}}+b^2(m-k)z^{\frac{2m-3k}{k}})dz
		\end{array}
		\end{displaymath}
		and the right side as
		\begin{displaymath}
		\begin{array}{l}
		(f(x)y-g(x))dx  \\
		=((a(2m+k)x^{m+k-1}+b(2m-k)x^{m-k-1})y\\
		-(a^2mx^{4k}+cx^{2k}+b^2m)x^{2m-2k-1})dx \bigskip\\
		
		=((a(2m+k)z^{\frac{m+k-1}{k}}+b(2m-k)z^{\frac{m-k-1}{k}})(az^{\frac{m+k}{k}}+tz^{\frac{m}{k}}+bz^{\frac{m-k}{k}})\\
		-(a^2mz^k+cz^2+b^2m)z^\frac{2m-2k-1}{k})\frac{1}{k}z^{\frac{1-k}{k}}dz\bigskip\\
		
		=((a(2m+k)z^{\frac{m+k-1-(k-1)}{k}}+b(2m-k)z^{\frac{m-k-1-(k-1)}{k}})(az^{\frac{m+k}{k}}+tz^{\frac{m}{k}}+bz^{\frac{m-k}{k}})z^{\frac{k-1}{k}}\\
		-(a^2mz^k+cz^2+b^2m)z^\frac{2m-2k-1}{k})\frac{1}{k}z^{\frac{1-k}{k}}dz\bigskip\\
		
		=((a(2m+k)z^{\frac{m}{k}}+b(2m-k)z^{\frac{m-2k}{k}})(az^{\frac{m+k}{k}}+tz^{\frac{m}{k}}+bz^{\frac{m-k}{k}})\\
		-(a^2mz^k+cz^2+b^2m)z^\frac{2m-2k-1}{k})\frac{1}{k}z^{\frac{1-k}{k}}dz\bigskip\\
		
		=(at(2m+k)z^{\frac{2m}{k}}+a^2(2m+k)z^{\frac{2m+k}{k}}+ab(2m+k)z^{\frac{2m-k}{k}}+bt(2m-k)z^{\frac{2m-2k}{k}}\\
		+ab(2m-k)z^{\frac{2m-k}{k}}+b^2(2m-k)z^{\frac{2m-3k}{k}}-a^2mz^{\frac{2m+k}{k}}
		-cz^{\frac{2m-k}{k}}-b^2mz^{\frac{2m-3k}{k}})\frac{1}{k}dz.\bigskip
		\end{array}
		\end{displaymath}
		
		For our purpose, we organize the terms with respect to $dz$, that is:
		\begin{displaymath}
		\begin{array}{l}
		(at(2m+k)z^{\frac{2m}{k}}+a^2(2m+k)z^{\frac{2m+k}{k}}+ab(2m+k)z^{\frac{2m-k}{k}}+bt(2m-k)z^{\frac{2m-2k}{k}} \bigskip \\
		+ab(2m-k)z^{\frac{2m-k}{k}}+b^2(2m-k)z^{\frac{2m-3k}{k}}-a^2mz^{\frac{2m+k}{k}}
		-cz^{\frac{2m-k}{k}}-b^2mz^{\frac{2m-3k}{k}}\bigskip \\ -t^2mz^{\frac{2m-k}{k}}-at(m+k)z^{\frac{2m}{k}}-bt(m-k)z^{\frac{2m-2k}{k}}-a^2(m+k)z^{\frac{2m+k}{k}}-atmz^{\frac{2m}{k}}\bigskip \\
		-ab(m-k)z^{\frac{2m-k}{k}}-tbmz^{\frac{2m-2k}{k}}-ab(m+k)z^{\frac{2m-k}{k}}-b^2(m-k)z^{\frac{2m-3k}{k}})dz\bigskip \\
		=(akz^2+tkz^+bk)z^{\frac{2m-k}{k}}dt
		\end{array}
		\end{displaymath}
		
		Now organizing the terms again we have:
		\begin{displaymath}
		\begin{array}{l}
		((ab(2m+k)+ab(2m-k)-c-t^2m-ab(m-k)-ab(m+k))z^{\frac{2m-k}{k}}\bigskip\\
		+(a^2(2m+k)-a^2m-a^2(m+k))z^{\frac{2m+k}{k}}+(at(2m+k)-at(m+k)-atm)z^{\frac{2m}{k}}\frac{dz}{dt}\bigskip\\
		=(akz^2+tkz+bk)
		\end{array}
		\end{displaymath}
		
		Thus, we obtain the Riccati equation:
		\begin{equation}
		(-mt^2+2mab-c)\frac{dz}{dt}=bk+tkz+akz^2 \label{lric}
		\end{equation}
		and we conclude the proof.
	\end{proof}

	
	For the rest of transformations proposed in the Exercise of Polyanin-Zaitsev we need the results concerning the transformations, which will be given in the next section.

	\section{Some transformations}
	In this section we study some transformations that allow us to complete the exercise stated by Polyanin-Zaitsev above.
	
	\begin{lemma}\label{lemma23}
		If $R=a_1x^2+a_1x+a_0$, $S=b_1x+b_0$, with $a_2,b_1\neq0$. The differential equation $$R^2\partial_x^2y+SR\partial_xy+Cy=0$$ is transformed in the equation  $$Q^2{\partial}_{\tau}^2\hat{y}+LQ{\partial}_{\tau}\hat{y}+\lambda \hat{y}=0,\quad Q=\tau^2+q_0,\quad L=l_1\tau+l_0, \quad \tau=x+\frac{a_1}{2a_2}.$$
	\end{lemma}
	
	\begin{proof}
		$R^2\partial_x^2y+SR\partial_xy+Cy=0$ then replacing
		$$(a_1x^2+a_1x+a_0)^2\partial_x^2y+(b_1x+b_0)(a_1x^2+a_1x+a_0)\partial_xy+Cy=0$$
		We divide all by $a_2^2$, thus we get:
		$$\left(x^2+\frac{a_1}{a_2}x+\frac{a_0}{a_2}\right)^2\partial_x^2y+\left(\frac{b_1}{a_2}x+\frac{b_0}{a_2}\right)\left(x^2+\frac{a_1}{a_2}x+\frac{a_0}{a_2}\right)\partial_xy+\frac{C
		}{a_2^2}y=0$$
		Now in $R$ we complete the square
		$$\left(x^2+\frac{a_1}{a_2}x+\left(\frac{a_1}{2a_2}\right)^2+\frac{a_0}{a_2}-\left(\frac{a_1}{2a_2}\right)^2\right)^2\partial_x^2y+$$ 
		$$\left(\frac{b_1}{a_2}x+\frac{b_0}{a_2}\right)\left(x^2+\frac{a_1}{a_2}x+\left(\frac{a_1}{2a_2}\right)^2+\frac{a_0}{a_2}-\left(\frac{a_1}{2a_2}\right)^2\right)\partial_xy+\frac{C
		}{a_2^2}y=0$$
		Then:
		$$\left(\left(x+\frac{a_1}{2a_2}\right)^2+\frac{a_0}{a_2}-\left(\frac{a_1}{2a_2}\right)^2\right)^2\partial_x^2y+$$
		$$\left(\frac{b_1}{a_2}x+\frac{b_0}{a_2}\right)\left(\left(x+\frac{a_1}{2a_2}\right)^2+\frac{a_0}{a_2}-\left(\frac{a_1}{2a_2}\right)^2\right)\partial_xy+\frac{C
		}{a_2^2}y=0.$$
		
		Assume $$\tau=x+\frac{a_1}{2a_2},\quad q_0=\frac{a_0}{a_2}-(\frac{a_1}{2a_2})^2,$$ then 
		$$(x+\frac{a_1}{2a_2})^2+\frac{a_0}{a_2}-(\frac{a_1}{2a_2})^2=\tau^2+q_0=Q(\tau).$$
		
		Replacing $x$ in the polynomial $S$ in term of $\tau$ we get: $$\frac{b_1}{a_2}x+\frac{b_0}{a_2}=\frac{b_1}{a_2}\left(\tau-\frac{a_1}{2a_2}\right)+\frac{b_0}{a_2}=\frac{b_1}{a_2}\tau-\frac{b_1.a_1}{2a_2^2}+\frac{b_0}{a_2}$$
		then if $l_1=\frac{b_1}{a_2}$ and $l_0=-\frac{b_1.a_1}{2a_2^2}+\frac{b_0}{a_2}$ we get $l_1x+l_0=L(\tau)$\\
		
		Now if $\lambda=\frac{C}{a_2^2}$, the differential equations will be:
		$$Q^2\hat{\partial}_{\tau}^2\hat{y}+LQ\hat{\partial}_{\tau}\hat{y}+\lambda \hat{y}=0$$
		where $\hat{y}=y(x(\tau))$, and the transformation  $\tau=x+\frac{a_1}{2a_2}$ send  $\partial_x$ on $\partial_{\tau}$.
	\end{proof}
	
	\begin{obs}
		The differential equation of the form $$(1-x^2)\partial_x^2y+(\tilde{b}-\tilde{a}-(\tilde{a}+\tilde{b}+2)x)\partial_xy+\lambda y=0$$
		with $\lambda=n(n+1+\tilde{a}+\tilde{b})$ and $n\in \mathbb{N}$, is known as Jacobi equation (in general form). It is a particular case of the hypergeometric equation, but the solutions include Jacobi polynomials. If we take $\tilde{a}=\tilde{b}$ and  $\tilde{\lambda}=n(n+2\tilde{a})$ with $n\in \mathbb{N}$, we get a Gegenbauer equation (or ultraspherical case):
		$$(1-x^2)\partial_x^2y-2(\tilde{a}+1)x\partial_xy+\tilde{\lambda} y=0.$$
	\end{obs}
	
	Now we study a special transformation, in the next theorem:\\
	\begin{teor}\label{teor1}
		The differential equation $$a_nz^{(n)}+(\tilde{k}+a_{n-1})z^{(n-1)}+a_{n-2}z^{(n-2)}+...+a_1z^{(1)}+a_0z=0,$$ with $a_i\in\mathbb{C}(x)$, with $a_n \neq 0$, can be transformed into the differential equation
		$$y^{(n)}+\tilde{k}y^{(n-1)}+b_{n-2}y^{(n-2)}+b_{n-3}y^{(n-3)}+...+b_1y^{(1)}+b_0y=0$$ with $b_i\in\mathbb{C}(x).$
	\end{teor}
	\bigskip
	
	\begin{proof}
		Following the Lemma \ref{lemma23} and taking the implicit transformation  $z=\varepsilon(t)y+\mu(t)$ with $\varepsilon(t),\mu(t) \in \mathbb{C}(x)$, we compute the first equation applying  the change of variable $$(\varepsilon y)^{(n)}+\mu^{(n)}+\Sigma_{i=0}^{n-2}b_i(\varepsilon y+\mu)^{(i)}=0.$$
		Then, computing in general form the Leibniz rule we get:

		\begin{flushright}
			$\varepsilon^{(n)}y+\binom{n-1}{1}(\varepsilon)^{(n-1)}y^{(1)}+\binom{n}{2}(\varepsilon)^{(n-2)}y^{(2)}+\binom{n}{3}(\varepsilon)^{(n-3)}y^{(3)}+...\binom{n}{n-1}\varepsilon^{(1)}y^{(n-1)}+\binom{n}{n}(\varepsilon)y^{(n)}+$\
			
			$(\tilde{k}+a_{n-1})\big[\varepsilon^{(n-1)}y+\binom{n-1}{1}(\varepsilon)^{(n-2)}y^{(1)}+\binom{n-1}{2}(\varepsilon)^{(n-3)}y^{(2)}+\binom{n-1}{3}(\varepsilon)^{(n-4)}y^{(3)}+...\binom{n-1}{n-2}\varepsilon^{(1)}y^{(n-2)}+\binom{n-1}{n-1}(\varepsilon)y^{(n-1)}\big]+$\
			
			$a_{n-2}\big[\varepsilon^{(n-2)}y+\binom{n-2}{1}(\varepsilon)^{(n-3)}y^{(1)}+\binom{n-2}{2}(\varepsilon)^{(n-4)}y^{(2)}+\binom{n-2}{3}(\varepsilon)^{(n-5)}y^{(3)}+...\binom{n-2}{n-3}\varepsilon^{(1)}y^{(n-3)}+\binom{n-2}{n-2}\varepsilon y^{(n-2)}\big]+$\
			
			$a_{n-3}\big[\varepsilon^{(n-3)}y+\binom{n-3}{1}(\varepsilon)^{(n-4)}y^{(1)}+\binom{n-3}{2}(\varepsilon)^{(n-5}y^{(2)}+\binom{n-3}{3}(\varepsilon)^{(n-6)}y^{(3)}+...\binom{n-3}{n-4}\varepsilon^{(1)}y^{(n-4)}+\binom{n-4}{n-4}\varepsilon y^{(n-3)}\big]+$...\
			
			$a_{n-m}\big[\varepsilon^{(n-m)}y+\binom{n-m}{1}(\varepsilon)^{(n-m-1)}y^{(1)}+\binom{n-m}{2}(\varepsilon)^{(n-m-2}y^{(2)}+\binom{n-m}{3}(\varepsilon)^{(n-m-3)}y^{(3)}+...\binom{n-m}{n-m-1}\varepsilon^{(1)}y^{(n-m-1}+\binom{n-m}{n-m}\varepsilon y^{(n-m)}\big]+... =0$\
			
		\end{flushright}
		
		Continuing of this form, we divide all equation by $\varepsilon$. Then, we use the same method of the indeterminate coefficients to calculate $\varepsilon$ and the $b_i$ coefficient:
		
		If we take  $y^{(n-1)}:$\
		
		$$\binom{n}{n-1}\frac{\varepsilon^{(1)}}{\varepsilon}+\tilde{k}+a_{(n-1)}=\tilde{k}$$
		then:
		$$ n\frac{\varepsilon^{(1)}}{\varepsilon}+a_{(n-1)}=0$$
		
		From this differential equation we get an appropriate $\varepsilon$ value, and with it we obtain the coefficient $b_{i}$.\\
		
		If we take $y^{(n-2)}:$\
		
		$$\binom{n}{n-2}\frac{\varepsilon^{(2)}}{\varepsilon}+a_{(n-1)}\binom{n-1}{n-2}\frac{\varepsilon^{(1)}}{\varepsilon}+a_{n-2} \binom{n-2}{n-2}=b_{n-2}$$
		
		If we take $y^{(n-3)}:$\
		
		$$\binom{n}{n-3}\frac{\varepsilon^{(3)}}{\varepsilon}+a_{n-1}\binom{n-1}{n-3} \frac{\varepsilon^{(2)}}{\varepsilon}+a_{n-2}\binom{n-2}{n-3}\frac{\varepsilon^{(1)}}{\varepsilon}+a_{n-3}\binom{n-3}{n-3}=b_{n-3}$$
		
		Continuing of this form we see, that for any  $k\in \mathbb{N}$ the recurrent formulae will be:\\
		
		$\binom{n}{n-k}\frac{\varepsilon^{(k)}}{\varepsilon}+...+a_{n-k+3}\binom{n-k+3}{n-k}\frac{\varepsilon^{(3)}}{\varepsilon}+a_{n-k+2}\binom{n-k+2}{n-k} \frac{\varepsilon^{(2)}}{\varepsilon}+a_{n-k+1}\binom{n-k+1}{n-k}\frac{\varepsilon^{(1)}}{\varepsilon}+a_{n-k}\binom{n-k}{n-k}=b_{n-k}$
	\end{proof}
	
	If in the previous theorem, we consider $\tilde{k}=0$, we get the next corollary:
	
	\begin{cor}\label{cor1}
		the differential equation $$a_nz^{(n)}+a_{n-1}z^{(n-1)}+a_{n-2}z^{(n-2)}+...+a_1z^{(1)}+a_0z=0$$ with $a_i\in\mathbb{C}(x)$ can transformed in the differential equation
		$$y^{(n)}+b_{n-2}y^{(n-2)}+b_{n-3}y^{(n-3)}+...+b_1y^{(1)}+b_0y=0$$ where $b_i\in\mathbb{C}(x)$
	\end{cor}
	
	\textbf{Example:}\
	Applying the transformation over the general second order differential equation $z''+a_1z'+a_0z=0$:\medskip
	
	Using the Theorem \ref{teor1}, we get   $\frac{2\varepsilon'}{\varepsilon}+a_1=0$, then $\varepsilon'=-\frac{a_1}{2}\varepsilon$. Now, through derivatives and dividing by $\varepsilon$, we get $\varepsilon''=-\frac{a'_1}{2}-\frac{\varepsilon'}{\varepsilon}\frac{a_1}{2}$, but $\frac{\varepsilon'}{\varepsilon}=-\frac{a_1}{2}$. Thus we obtain
	$$\varepsilon''=-\frac{a'_1}{2}-\frac{a_1^2}{2}.$$ In this way we arrive to $b=-\frac{a'_1}{2}-\frac{a_1^2}{4}+a_0$, for the differential equation $y''+b=0$
	
	In the following theorem we recall that a Hamiltonian Change of Variable $z=z(x)$ is a change of variable in where $(z(x),z'(x))$ is a solution curve of a Hamiltonian system of one degree of freedom. The new derivation is given by $\hat\partial_z=\sqrt{\alpha}\partial_z$, being $\alpha=(\partial_xz)^2$, see \cite{Ac3,AB,Amw} and references therein. 
	\begin{teor}
		Let $Q$ and $L$ be as in \eqref{lemma23}, with $a_1=b_1=0$ or $b_0=\frac{a_1b_1}{2a_2}$. Through the Hamiltonian change of variable $\xi=\partial_{\tau}\sqrt{Q}$, the differential equation $$Q^2\partial_t^2w+LQ\partial_tw+\lambda w=0,$$ is transformed in the equation $$(1-\xi^2)\partial_{\tau}^2\hat{w}+\left(l_1-\frac{3}{q_0}\right)\xi\partial_{\tau}\hat{w}+\frac{\lambda}{q_0}\hat{w}=0.$$ Owing to  $\lambda=n(n+1+\tilde{a}+\tilde{b})$, we have the Jacobi equation with $\tilde{a}=\tilde{b}$. Moreover, if $\lambda=n(n+2\tilde{a})$ then we obtain a Gegenbauer equation.
	\end{teor}
	
	\begin{proof}
		\textbf{Case 1}:\\
		If we assume $a_1=b_0=0$, we obtain $l_0=0$, $l_1=\frac{b_1}{a_2}$, $q_0=\frac{a_0}{a_2}$.
		Then:
		$$(\tau^2+q_0)^2\partial_\tau^2w+l_1\tau(\tau^2+q_0)\partial_\tau w+\lambda w=0.$$
		Now, by hypothesis $\xi=\partial_{\tau}\sqrt{Q}$, this is
		\begin{displaymath}
		\xi=\frac{\partial_{\tau} Q}{2\sqrt{Q}}=\frac{\partial_{\tau}(\tau^2+q_0)}{2\sqrt{\tau^2+q_0}}\\
		\xi^2=\frac{\tau^2+q_0}{\tau^2+q_0}-\frac{q_0}{\tau^2+q_0}\\
		\frac{q_0}{1-\xi^2}=\tau^2+q_0\\
		\tau^2=\frac{q_0}{1-\xi^2}-q_0\\
		\tau=\pm\frac{\xi\sqrt{q_0}}{\sqrt{1-\xi^2}}
		\end{displaymath}
		Furthermore, due to $\xi=\partial_{\tau}\sqrt{Q}$ then $\alpha$ we arrive to:
		\begin{displaymath}
		\sqrt{\alpha}=\partial_{\tau}\xi=\partial_{\tau}^2\sqrt{Q}\\
		=\frac{1}{\frac{\sqrt{q_0}(\sqrt{1-\xi^2}-\xi.\frac{2\xi}{2\sqrt{1-\xi^2}})}{1-\xi^2}}\\
		=\frac{1}{\frac{\sqrt{q_0}(1-\xi^2+\xi^2)}{\sqrt{1-\xi^2}(1-\xi^2)}}\\
		=\sqrt{\frac{(1-\xi^2)^3}{q_0}}
		\end{displaymath}
		i.e $\alpha=\frac{(1-\xi^2)^3}{q_0}$.\\
		
		Now $\hat{\partial}_{\xi}=\sqrt{\frac{(1-\xi^2)^3}{q_0}}\partial_{\xi}$, and\\   $\hat{\partial}_{\xi}^2=\alpha\partial_\xi^2+\frac{1}{2}\partial_\xi\alpha\partial_\xi=\frac{(1-\xi^2)^3}{q_0}\partial_\xi^2-\frac{3(1-\xi^2)^2}{q_0}\partial_\xi$.\\
		
		We compute all elements of the Hamiltonian change of variable $\hat{Q}=Q(\tau(\xi))=\frac{q_0^2}{1-\xi^2}$, $\hat{L}=l_1\frac{\xi\sqrt{q_0}}{\sqrt{1-\xi^2}}$, $\hat{w}(\xi(\tau))=w(\tau)$.\\
		Then:
		\begin{displaymath}
		Q^2\partial_\tau^2 w\\
		\rightsquigarrow \hat{Q}^2\hat{\partial}_\xi^2\hat{w}
		=(\frac{q_0^2}{1-\xi^2})^2(\frac{(1-\xi^2)^3}{q_0}\partial_\xi^2\hat{w}-\frac{3(1-\xi^2)^2}{q_0}\partial_\xi\hat{w})\\
		=q_0(1-\xi^2)\partial_\xi^2\hat{w}-3\xi\partial_\xi\hat{w}
		\end{displaymath}
		
		\begin{displaymath}
		LQ\partial_\tau w\\
		\rightsquigarrow \hat{L}\hat{Q}\hat{\partial}_\xi\hat{w}
		=(l_1\frac{\xi\sqrt{q_0}}{\sqrt{1-\xi^2}})(\frac{q_0^2}{1-\xi^2})\sqrt{\frac{(1-\xi^2)^3}{q_0}}\partial_{\xi}\hat{w}
		=l_1q_0\xi\partial_{\xi}\hat{w}
		\end{displaymath}
		
		If we replace, in the transformed differential equation, we obtain:
		$$q_0(1-\xi^2)\partial_\xi^2\hat{w}+(l_1q_0-3)\xi\partial_{\xi}\hat{w}+\lambda\hat{w}=0$$
		It is equivalent to Gegenbauer equation with $\tilde{\lambda}=\frac{\lambda}{q_0}$, $-2\tilde{a}+1=l_1-\frac{3}{q_0}$ i.e $\tilde{a}=\frac{\frac{3}{q_0}-l_1-1}{2}$ then:
		$$(1-\xi^2)\partial_\xi^2\hat{w}+\left(l_1-\frac{3}{q_0}\right)\xi\partial_{\xi}\hat{w}+\frac{\lambda}{q_0}\hat{w}=0$$
		
		Finally if $\frac{\lambda}{q_0}=n(n+\frac{3}{q_0}-l_1-1)$ with $n\in\mathbb{N}$, then the solutions of this equation are Ultraspherical Gegenbauer polynomial.\\
		
		\textbf{Case 2:}\\
		If $b_0=\frac{a_1b_1}{2a_2}$ then:\\
		$Q(\tau)=\tau+q_0$ with $q_0=\frac{a_0}{a_2}-(\frac{a_1}{2a_2})^2$, $L(\tau)=\tilde{l_1}\tau+\tilde{l_0}$, $\tilde{l_1}=\frac{b_1}{a_2}$ and $\tilde{l_0}=-\frac{b_1a_1}{2a_2^2}+\frac{b_0}{a_2}=-\frac{b_1a_1}{2a_2^2}+\frac{b_1a_2}{2a_2^2}=0$\\
		i.e the initial differential equation it will be:
		
		$$(\tau^2+q_0)^2\partial_\tau^2\hat{w}+\tilde{l_1}\tau(\tau^2+q_0)\partial_\tau\hat{w}+\lambda\tilde{w}=0 $$
		for instance, the same differential equation of the previous case.
	\end{proof}
	
	Now we transform the Gegenbauer  equation into an Hypergeometric equation.
	\begin{lemma}
		The Gegenbauer equation $$(1-x^2)\partial_x^2 y-2(\mu+1)x\partial_x y+(\nu-\mu)(\nu+\mu+1)y=0$$ is transformed into an hypergeometric equation  $$z(1-z)\partial_z^2y+(c-(a+b+1)z)\partial_zy-aby=0,$$ where $a=\mu-\nu$, $b=\nu+\mu+1$ and $c=\mu+1.$
	\end{lemma}
	
	\begin{proof}
		For this transformation we will use a Hamiltonian change of variable, over the independent variable of the Gegenbauer equation $x=1-2z$. Then $\sqrt{\alpha}=\partial_xz=-\frac{1}{2}$, that is,  $\alpha=\frac{1}{4}$, $\hat{\partial_z}=\frac{1}{2}\partial_z$ and $\hat{\partial_z^2}=\frac{1}{4}\partial_z^2$.\\
		Substituting in the Gegenbauer equation we obtain
		$$(1-x^2)\partial_x^2y=(1-(1-4z+4z^2))\frac{1}{4}\hat{\partial}_z^2\hat{y}=z(1-z)\hat{\partial}_x^2\hat{y}$$
		Thus, we obtain the equation
		$$z(1-z)\hat{\partial}_x^2\hat{y}+(\mu+1)(1-2z)\hat{\partial}_x\hat{y}-(\mu-\nu)(\nu+\mu+1)\hat{y}=0.$$
		
		We know that Hypergeometric equation is of the form $$z(1-z)\partial_z^2y+(c-(a+b+1)z)\partial_zy-aby=0.$$ Then we compute the parameters values  $a,\ b\ y\ c$ as follows:\\
		
		\begin{displaymath}
		\begin{array}{l}
		ab=(\mu-\mu)(\mu+\nu+1)\bigskip\\
		c-(a+b+1)z=(\mu+1)-2(\mu+1)z
		\end{array}
		\end{displaymath}
		
		Therefore $$c=\mu+1,\quad a=\mu-\nu,\quad b=\mu+\nu+1,$$
		which concludes the proof.	
	\end{proof}

	\begin{prop}\label{prop2}
		Through the Hamiltonian change of variables $x=1-2\xi$ and $y=(x^2-1)^{\frac{\mu}{2}}$, we can transform the Hypergeometric equation  $$\xi(\xi-1)\partial_{\xi}^2w+(\mu+1)(1-2\xi)\partial_{\xi}w+(\nu-\mu)(\nu+\mu+1)w=0$$ into the Legendre equation $$(1-x^2)\partial_x^2y-2x(1-x^2)\partial_xy+[\nu(\nu+1)(1-x^2)-\mu^2]y=0.$$ 
	\end{prop}
	\begin{proof}
		Firstly we transform the Legendre equation into the Hypergeometric equation:\bigskip
		
		If  $y=(x^2-1)^{\frac{\mu}{2}}$ then
		$$\partial_xy=\mu x(x^2-1)^{\frac{\mu}{2}-1}w+(x^2-1)^{\frac{\mu}{2}}\partial_{x}w$$
		and
		$$\partial_{x}^{2}y=\mu [(x^2-1)^{\frac{\mu}{2}-1}+2(\frac{\mu}{2}-1)x^2(x^2-1)^{\frac{\mu}{2}-2}]w+2\partial
		x(x^2-1)^{\frac{\mu}{2}-1}\partial_x w+(x^2-1)^{\frac{\mu}{2}}\partial_x^2 w$$
		Now dividing by $(x^2-1)^{\frac{\mu}{2}-2}$ and replacing in each terms of the Legendre equation we get:\\
		$(x^2-1)^2\partial_x^2 y=\mu(x^2-1)w+2\mu(\frac{\mu}{2}-1)x^2 w+2\mu x(x^2-1)\partial_x w+(x^2-1)^2\partial_x^2 w$\bigskip\\
		$2x(x^2-1)\partial_x y=2\mu x^2 w+2x(x^2-1)\partial_x w$\bigskip\\
		$[v(v+1)(1-x^2)-\mu^2]y=[-v(v+1)(x^2-1)-\mu^2]w$\bigskip\\
		Now we obtain\\
		\begin{displaymath}
		\begin{array}{l}
		(x^2-1)^2\partial_x^2 w+[2\mu x+2x](x^2-1)\partial_x w+\\
		(-v(v+1)(x^2-1)-\underbrace{\mu^2+2\mu x^2}+\mu(x^2-1)+2\mu(\frac{\mu}{2}-1)x^2)w=0\\
		\Rightarrow\\
		(x^2-1)^2\partial_x^2 w+2x(\mu+1)(x^2-1)\partial_x w+\\
		(-v(v+1)(x^2-1)-\underbrace{\mu^2+2\mu x^2+\mu x^2-\mu+\mu^2 x^2-2\mu x^2})=0\\
		\Rightarrow\\
		(x^2-1)^2\partial_x^2 w+2x(\mu+1)(x^2-1)\partial_x w+\\
		(-v(v+1)(x^2-1)+\mu^2(x^2-1)+\mu(x^2-1))w=0\\
		\Rightarrow\\
		(x^2-1)\partial_x^2 w+2x(\mu+1)\partial_x w+[\mu^2-v^2+\mu-v]w=0 \bigskip\\
		\Rightarrow\\
		(x^2-1)\partial_x^2 w+2x(\mu+1)\partial_x w+(\mu-v)(\mu+v+1)w=0
		\end{array}
		\end{displaymath}
		
		Now applying the Hamiltonian change of variable 
		$x=1-2\xi$ $\xi=\frac{1-x}{2}$, we obtain  $\partial_x\xi=\frac{-1}{2}=\sqrt{\alpha}$, being $\alpha=\frac{1}{4}$. Thus, 
		we obtain $\hat\partial_\xi=\frac{-1}{2}\partial_\xi$,  then
		$x^2-1=1-4\xi+4\xi^2-1=4\xi(\xi-1)$. Now, replacing we obtain:\\
		$$\xi(\xi-1)\partial_\xi\hat{w}-(\mu+1)(1-2\xi)\partial_\xi w+(\mu-v)(\mu+v+1)\hat{w}=0$$
		
		Following exactly the reversed process we can transform a Hypergeometric equation into Legendre equation.
		\end{proof}
	\begin{Ejem}
		Transform the next equation on ultraespheric form.
		$$(mt^2+c_0)^2 \ddot{w}+t(2m+k)(mt^2+c_0) \dot{w}+abk^2 w=0.$$
		Now  $R=mt^2+c_0$, $S=(2m+k)t$ $\tau=t+\frac{c_0}{2m}$, $q_0=\frac{c_0}{m}$, $l_0=0$, $l_1=\frac{2m+k}{m}$ and $\lambda=\frac{}abk^2{m^2}$.\\
		
		Applying the previous lemma we have the equation:
		$$(\tau^2+\frac{c_0}{m})^2\partial_\tau^2\hat{w}+\frac{2m+k}{m}\tau(\tau^2+\frac{c_0}{m})\partial_\tau\hat{w}+\frac{}abk^2{m^2}\tilde{w}=0.$$
		
		Now applying the previous theorem we get:
		$$(1-\xi^2)\partial_\xi^2\hat{u}+(\frac{2m+k}{m}-\frac{3m}{c_0})\xi\partial_{\xi}\hat{u}+\frac{abk^2}{c_0m}\hat{u}=0$$
		
		with $\hat{u}(\xi)=\hat w(\tau(\xi))$
	\end{Ejem}
	
	\begin{obs}
		If we have our equation in the Legendre form, we apply the proposition \ref{prop2} and therefore we can study it as in \cite{Almp} to conclude the integrability or non-integrability, of the Li\'enard equation. Moreover, such as we will see in the next section, through equation \eqref{din1} in Legendre's form  we can apply the Kimura table, see \cite{Almp}.
	\end{obs}
	
	\section{Polyanin-Zaitsev vector field}
	The associated system of the Polyanin-Zaitsev vector field, with  $a,b,c,m,k\in\mathbb{R}$, is given by:
	
	\begin{equation}\label{din1}
	\begin{array}{ccl}
	\dot{x}&=&y \bigskip \\
	\dot{y}&=& (\alpha x^{m+k-1}+\beta x^{m-k-1})y-\gamma x^{2m-2k-1}.
	\end{array}
	\end{equation}
	
	with $\alpha=a(2m+k)$, $\beta=b(2m-k)$ and $\gamma=(a^2mx^{4k}+cx^{2k}+b^2m)$,\\
	where the Polyanin-Zaitsev vector field is given by $X=:(P,Q)$, being, $$(P,Q):=(y,(a(2m+k)x^{m+k-1}+b(2m-k)x^{m-k-1})y-(a^2mx^{4k}+cx^{2k}+b^2m)x^{2m-2k-1}).$$
	
	The next proposition can illustrate the cases in which the Polyanin-Zaitsev vector field is formed by non trivial polynomial functions.\\
	
	\begin{prop}
		The system \eqref{din1} is a not null differential polynomial system if it is equivalently to one of the next families:
		
		\begin{equation}\label{F1}
		\begin{array}{ccl}
		\dot{x}&=&y\\
		\dot{y}&=&[a\frac{3s+p+4}{2} x^s+b\frac{s+3p+4}{2} x^p]y-a\frac{s+p+2}{2} x^{2s+1}-c x^{s+p+1}-b^2\frac{s+p+2}{2} x^{2p+1}
		\end{array}	
		\end{equation}
		
		\begin{equation}\label{F2}
		\begin{array}{ccl}
		\dot{x}&=&y\\
		\dot{y}&=&b\frac{r+2p+3}{2} yx^p-c x^r-b^2\frac{r+1}{2} x^{2p+1}
		\end{array}	
		\end{equation}
		
		\begin{equation}\label{F3}
		\begin{array}{ccl}
		\dot{x}&=&y\\
		\dot{y}&=&[a\frac{3s+p+4}{2} x^s+b\frac{s+3p+4}{2} x^p]y-a\frac{s+p+2}{2} x^{2s+1}-b^2\frac{s+p+2}{2} x^{2p+1}
		\end{array}	
		\end{equation}
		
		\begin{equation}\label{F4}
		\begin{array}{ccl}
		\dot{x}&=&y\\
		\dot{y}&=&-c x^{s+p+1}
		\end{array}	
		\end{equation}

		\begin{equation}\label{F5}
		\begin{array}{ccl}
		\dot{x}&=&y\\
		\dot{y}&=&a\frac{3s+p+4}{2} yx^s-a\frac{r+1}{2} x^{2s+1}-c x^r
		\end{array}	
		\end{equation}
		
		\begin{equation}\label{F6}
		\begin{array}{ccl}
		\dot{x}&=&y\\
		\dot{y}&=&b(m+p+1) yx^p-b^2m x^{2p+1}
		\end{array}	
		\end{equation}

		\begin{equation}\label{F7}
		\begin{array}{ccl}
		\dot{x}&=&y\\
		\dot{y}&=&a(m+s+1)yx^s-am x^{2s+1}
		\end{array}	
		\end{equation}
	
	with $s,p,r \in \mathbb{Z}$ defined in the proof.
	\end{prop}
	
	\begin{proof}
				The system \eqref{din1} is a polynomial system if $Q$ is a polynomial function, that is, the exponents of each term must be non negative integer. Furthermore, we need to consider the values of the constants $a$, $b$ and $c$. Now we consider the different possibilities for these constants:
		
		\begin{enumerate}
			\item[Case 1.] For $a\neq 0$,  $b\neq 0$, $c\neq 0 $, it must be satisfied: 
			\begin{displaymath}
			\begin{array}{l}
			m+k-1=s\\
			m-k-1=p\\
			2m+2k-1=2s+1\\
			2m-1=r\\
			2m-2k-1=2p+1,
			\end{array}
			\end{displaymath}
			being $s,p,r\in \mathbb{Z}^{+}$. Thus $m=\frac{r+1}{2}$ and $m=\frac{s+p+2}{2}$, which means that $r=s+p+1$. Therefore we obtain the following system associated with the family \eqref{din1}:
			
			\begin{equation*}
			\begin{array}{ccl}
			\dot{x}&=&y\\
			\dot{y}&=&[a\frac{3s+p+4}{2} x^s+b\frac{s+3p+4}{2} x^p]y-a\frac{s+p+2}{2} x^{2s+1}-c x^{s+p+1}-b^2\frac{s+p+2}{2} x^{2p+1}.
			\end{array}	
			\end{equation*}

			\item[Case 2.] For $a=0$,  $b\neq 0$, $c\neq 0$, the system \eqref{din1} is reduced to:
			
			\begin{displaymath}
			\begin{array}{ccl}
			\dot{x}&=&y\\
			\dot{y}&=&b(2m-k)yx^{m-k-1}-b^2mx^{2m-2k-1}-cx^{2m-1},
			\end{array}
			\end{displaymath}
			
			since the exponents must be non-negative integers, then:
			\begin{displaymath}
			\begin{array}{l}
			m-k-1=p\\
			2m-1=r\\
			2m-2k-1=2p+1,
			\end{array}
			\end{displaymath}
			for instance, we arrive to the system:
			\begin{equation*}
			\begin{array}{ccl}
			\dot{x}&=&y\\
			\dot{y}&=&b\frac{r+2p+3}{2} yx^p-c x^r-b^2\frac{r+1}{2} x^{2p+1}.
			\end{array}	
			\end{equation*}
			
			\item[Case 3.] For $a\neq 0$,  $b\neq 0$, $c=0$, the system \eqref{din1} is reduced to:
			
			\begin{displaymath}
			\begin{array}{ccl}
			\dot{x}&=&y \bigskip \\
			\dot{y}&=&a(2m+k)yx^{m+k-1}+b(2m-k)yx^{m-k-1}-a^2mx^{2m+2k-1}-b^2mx^{2m-2k-1},
			\end{array}
			\end{displaymath}
			
			again the exponents must be non-negative integers, therefore:
			\begin{displaymath}
			\begin{array}{l}
			m+k-1=s\\
			m-k-1=p\\
			2m+2k=2s+1\\
			2m-2k-1=2p+1,
			\end{array}
			\end{displaymath}
			
			thus, $m=\frac{s+p+2}{2}$ and $k=\frac{s-p}{2}$, which lead us to the following system:
			\begin{equation*}
			\begin{array}{ccl}
			\dot{x}&=&y\\
			\dot{y}&=&b\frac{s+3p+4}{2} yx^p-c x^r-b^2\frac{r+1}{2} x^{2p+1}.
			\end{array}	
			\end{equation*}

			\item[Case 4.] For $a=0$,  $b=0$, $c\neq 0$, the system \eqref{din1} is reduced to:
			
			\begin{displaymath}
			\begin{array}{ccl}
			\dot{x}&=&y \bigskip \\
			\dot{y}&=&cx^{2m-1},
			\end{array}
			\end{displaymath}
			
			due to the exponents must be non-negative integers, we arrive to $2m-1=r \in \mathbb{N}$, that is, $m=\frac{r+1}{2}$. 
			
			\begin{equation*}
			\begin{array}{ccl}
			\dot{x}&=&y\\
			\dot{y}&=&-c x^{s+p+1}.
			\end{array}	
			\end{equation*}

			\item[Case 5.]  For $a\neq 0$,  $b=0$, $c\neq 0$,  the system \eqref{din1} is reduced to:
			
			\begin{displaymath}
			\begin{array}{ccl}
			\dot{x}&=&y \bigskip \\
			\dot{y}&=&a(2m+k)yx^{m+k-1}-a^2mx^{2m+2k-1}-cx^{2m-1},
			\end{array}
			\end{displaymath}
			
			in this case:
			
			\begin{displaymath}
			\begin{array}{l}
			m+k-1=s\\
			2m+2k-1=2s+1\\
			2m-1=r,\\
			\end{array}
			\end{displaymath}
			
			then $m=\frac{r+1}{2}$ and $k=\frac{2s-r+1}{2}$. Thus we arrive to the system:
			
			\begin{equation*}
			\begin{array}{ccl}
			\dot{x}&=&y\\
			\dot{y}&=&a\frac{3s+p+4}{2} yx^s-a\frac{r+1}{2} x^{2s+1}-c x^r.
			\end{array}	
			\end{equation*}

			\item[Case 6.]  For $a= 0$,  $b\neq 0$, $c=0$, the system \eqref{din1} is reduced to:
			
			\begin{displaymath}
			\begin{array}{ccl}
			\dot{x}&=&y \bigskip \\
			\dot{y}&=&b(2m-k)yx^{m-k-1}-b^2mx^{2m-2k-1},
			\end{array}
			\end{displaymath}
			
			in this case:
			\begin{displaymath}
			\begin{array}{l}
			m-k-1=p\\
			2m-2k-1=2p+1,
			\end{array}
			\end{displaymath}
			then there is a line of solutions,  with  $r\in \mathbb{Z}^{+}$.
			
			In this case the associated family is: 
			\begin{equation*}
			\begin{array}{ccl}
			\dot{x}&=&y\\
			\dot{y}&=&+b(m+p+1) yx^p-b^2m x^{2p+1}.
			\end{array}	
			\end{equation*}

			\item[Case 7.] For $a\neq 0$,  $b=0$, $c=0$, the system \eqref{din1} is reduced to:
			
			\begin{displaymath}
			\begin{array}{ccl}
			\dot{x}&=&y \bigskip \\
			\dot{y}&=&a(2m+k)yx^{m+k-1}-a^2mx^{2m+2k-1},
			\end{array}
			\end{displaymath}
			
			in this case:
			
			\begin{displaymath}
			\begin{array}{l}
			m+k-1=s\\
			2m+2k-1=2s+1,
			\end{array}
			\end{displaymath}
			
			then the associated family is:
			
			\begin{equation*}
			\begin{array}{ccl}
			\dot{x}&=&y\\
			\dot{y}&=&a(m+s+1)yx^s-am x^{2s+1}.
			\end{array}	
			\end{equation*}
				\end{enumerate}
	\end{proof}

	\section{Finite critical points}
	
	In this section, we present an study about the existence of finite critical points and  the  stability for each family associated to the Polyanin-Zaitsev vector field.
	
	\begin{prop}
				For the family of systems \eqref{F1} the following statements hold:
		\begin{enumerate}
			\item[i.] If $c>0$, then $(0,0)$, is the only one finite critical point of the family.
			\begin{enumerate}
				\item[a.] If $k=0$ and $m\geq1$, then $(0,0)$ is an stable critical point.
				\item[b.] If $m+k-1$ is even and $a(2m+k)>0$, then $(0,0)$ is an unstable node.
				\item[c.] If $m+k-1$ is odd, then $(0,0)$ is the union of one elliptic sector with one hyperbolic sector.
			\end{enumerate}
			\item[ii] If $c<0$ then exist five finite critical points. 
			
		\end{enumerate}
			\end{prop}
	
	\begin{proof}
		For this proof we take the family \eqref{F1} in form \eqref{din1}.
		That is, we have to find solutions of the system:
		
		\begin{displaymath}
		\begin{array}{rcl}
		y&=&0 \bigskip \\
		(a(2m+k)x^{m+k-1}+b(2m-k)x^{m-k-1})y-(a^2mx^{4k}+cx^{2k}+b^2m)x^{2m-2k-1}&=&0
		\end{array}
		\end{displaymath}
		If $y=0$, then $(a^2mx^{4k}+cx^{2k}+b^2m)x^{2m-2k-1}=0$, with $2m-2k-1 \geq 1$, then for the product equal to $0$ it must be fulfilled that  $x=0$ or $(a^2mx^{4k}+cx^{2k}+b^2m)=0$. In the first case we obtain that $x=0$ and we conclude $(x,y)=(0,0)$.\\
		
		Now completing squares in the polynomial:\\
		\begin{displaymath}
		\begin{array}{l}
		a^2mx^{4k}+cx^{2k}+b^2m\bigskip\\
		=a^2m(x^{4k}+\frac{cx^{2k}}{a^2m}+\frac{b^2m}{a^2m})\bigskip\\
		=a^2m(x^{4k}+\frac{cx^{2k}}{a^2m}+\frac{c^2}{4a^4m^2} +\frac{b^2m}{a^2m}-\frac{c^2}{4a^4m^2})\bigskip\\
		=a^2m(x^{2k}+\frac{c}{2a^2m})^2 +a^2m(\frac{b^2m}{a^2m}-\frac{c^2}{4a^4m^2})\bigskip\\
		=a^2m(x^{2k}+\frac{c}{2a^2m})^2 -(\frac{c^2-4a^2b^2m^2}{4a^2m})\bigskip\\
		=a^2m((x^{2k}+\frac{c}{2a^2m})^2 -(\frac{c^2-4a^2b^2m^2}{4a^4m^2}))\bigskip\\
		=a^2m(x^{2k}+\frac{c}{2a^2m} +\frac{\sqrt{c^2-4a^2b^2m^2}}{2a^2m})(x^{2k}+\frac{c}{2a^2m}-\frac{\sqrt{c^2-4a^2b^2m^2}}{2a^2m})\bigskip\\
		=a^2m(x^{2k}+\frac{c+\sqrt{c^2-4a^2b^2m^2}}{2a^2m})(x^{2k}+\frac{c-\sqrt{c^2-4a^2b^2m^2}}{2a^2m})\bigskip\\
				\end{array}
		\end{displaymath}
		We can see that if $c>0$ for the equations $x^{2k}+\frac{c+\sqrt{c^2-4a^2b^2m^2}}{2a^2m}=0$ or\\ $x^{2k}+\frac{c-\sqrt{c^2-4a^2b^2m^2}}{2a^2m}=0$, then there are not real roots. That is, the only finite critical point is $(0,0)$.\\
		
		Now if $c<0$ we have that $y=0$ and  $(a^2mx^{4k}+cx^{2k}+b^2m)x^{2m-2k-1}=0$. Again $(0,0)$ is a first critical point. But there are other solutions for equations $a^2mx^{4k}+cx^{2k}+b^2m=0$, 
		$$x_1^{2k}=\frac{-c+\sqrt{c^2-4a^2b^2m^2}}{2a^2m} \qquad x_2^{2k}=\frac{-c-\sqrt{c^2-4a^2b^2m^2}}{2a^2m}$$
		We remember, in this case $c<0$ then $x_1^{2k}$ and $x_2^{2k}$ it is always positive, then: 
		
		$$x_1^{k}=\pm \sqrt{\frac{-c+\sqrt{c^2-4a^2b^2m^2}}{2a^2m}} \qquad x_2^{k}=\pm \sqrt{\frac{-c-\sqrt{c^2-4a^2b^2m^2}}{2a^2m}}$$
		
		Now we have to consider two cases:
		\begin{itemize}
			\item[$k\in\mathbb{Z}$:] Then the critical points for \eqref{F1} are $$(0,0), \quad \left(\pm \sqrt[2k]{\frac{-c+\sqrt{c^2-4a^2b^2m^2}}{2a^2m}},0\right),\quad \left(\pm \sqrt[2k]{\frac{-c-\sqrt{c^2-4a^2b^2m^2}}{2a^2m}},0\right).$$
			
			\item[ $k\notin \mathbb{Z}$:] The critical points are 
			$$(0,0), \quad \left( \sqrt[s-p]{\frac{-c+\sqrt{c^2-4a^2b^2m^2}}{2a^2m}},0\right),\quad \left( \sqrt[s-p]{\frac{-c-\sqrt{c^2-4a^2b^2m^2}}{2a^2m}},0\right).$$
		\end{itemize}
		
		For the next step we consider $c>0$ and the Liapunov function  $$V=c_1x^{2m-2k}+c_2x^{m-k}y+c_3y^2$$
		The next stage is to find the conditions in which this function can be positive:
		
		\begin{displaymath}
		\begin{array}{l}
		V=c_1x^{2m-2k}+c_2x^{m-k}y+c_3y^2\bigskip\\
		=c_1(x^{2m-2k}+\frac{c_2yx^{m-k}}{c_1}+\frac{c_3y^2}{c_1})\bigskip\\
		=c_1(x^{2m-2k}+\frac{c_2yx^{m-k}}{c_1}+\frac{c_2^2y^2}{4c_1^2}+\frac{c_3y^2}{c_1}-\frac{c_2^2y^2}{4c_1^2})\bigskip\\
		=c_1(x^{m-k}+\frac{c_2y}{c_1})^2+(\frac{4c_1c_3-c_2^2}{4c_1})y^2\bigskip\\
		\end{array}
		\end{displaymath}
		Therefore, the function $V$ is positive for $c_1>0$ and $4c_1c_3-c_2^2\geq0$.\\
		The derivative of $V$ is $$V'=2(m-k)c_1x'x^{2m-2k-1}+c_2(m-k)yx'x^{m-k-1}+c_2x^{m-k}y'+2c_3yy'.$$
		
		Now for the family \eqref{din1},  $\dot{y}=2m(a+b)x^{m-1}y-(a^2m+c+b^2m)x^{2m-1}$
		we have that
		$V'=2mc_1(a+b)yx^{2m-1}+c_2my^2x^{m-1}+c_2myx^{2m-1}-c_2(a^2m+c+b^2m)x^{3m-1}\\
		+4mc_3(a+b)x^{m-1}y^2-2c_3(a^2m+c+b^2m)yx^{2m-1}$.\\
		Arranging the right side we get:
		$$V'=yx^{2m-1}(2mc_1(a+b)+mc_2-2c_3(a^2m+c+b^2m))+my^2 x^{m-1}(c_2+4c_3(a+b))-c_2(a^2m+c+b^2m)$$
		We can observe that two cases should be considered, the first one corresponds to $m$ is odd. Thus, the critical point $ (0,0)$ is stable whenever:
		
		\begin{enumerate}
			\begin{multicols}{3}
				\item[a.]$a^2m+c+b^2m=0$
				\item[b.]$2mc_1(a+b)+mc_2=0$
				\item[c.]$c_2+4c_3(a+b)<0$
			\end{multicols}
		\end{enumerate}
		
		If $m$ is even, for $ (0,0)$ to be stable, it is required that:\\
		\begin{enumerate}
			\begin{multicols}{3}
				\item[a.]$a^2m+c+b^2m=0$
				\item[b.]$2mc_1(a+b)+mc_2=0$
				\item[c.]$c_2+4c_3(a+b)<0$
			\end{multicols}
		\end{enumerate}

		Now over the conditions of \eqref{teorest}:\\
		$(0,0)$ is an isolated critical point.\\
		\begin{displaymath}
		\begin{array}{l}
		X(x,y)=0\\
		Y(x,y)=(a(2m+k)x^{m+k-1}+b(2m-k)x^{m-k-1})y-(a^2mx^{4k}+cx^{2k}+b^2m)x^{2m-2k-1}\\
		\end{array}
		\end{displaymath}
		The degree of $Y(x,y)$ should be greater than $1$.
		
		\begin{displaymath}
		\begin{array}{l}
		y=F(x)=0\bigskip\\
		f(x)=Y(x,F(x))-(a^2mx^{4k}+cx^{2k}+b^2m)x^{2m-2k-1}
		=-a^2mx^{2m+2k-1}(1+\frac{c}{a^2m}x^{-2k}+\frac{b^2}{a^2}x^{4k})\bigskip\\
		\phi(x)=(\frac{\partial X}{\partial x}+\frac{\partial Y}{\partial y})|_{(x,F(x))}=a(2m+k)x^{m+k-1}(1+\frac{b(2m-k)}{a(2m+k)}x^{-2k})\bigskip\\
		Then\\
		\alpha=2m+2k-1\\
		\beta=m+k-1\\
		\bar{a}=-a^2m\\
		\bar{b}=a(2m+k)
		\end{array}
		\end{displaymath}
		
		Now checking the conditions of the theorem we have:\\
		$\alpha$ is odd, $\bar{a}<0$,\\
		$\bar{b}^2+4\bar{a}(\beta+1)=a^2(2m+k)+4a^2m(m+k)>0$, we have the conditions of item \emph{c)} \eqref{teorest}.\\
		If $\beta$ is even and $\bar{b}>0$, then $(0,0)$ is an unstable node. On the other hand, if $\beta$ is odd, then there exists the union of an elliptical sector and with an  hyperbolic sector.	Thus, we conclude the proof.
	\end{proof}

	\begin{prop}
		For the system \eqref{F2} and \eqref{F5} there are three critical points 
	\end{prop}
	\begin{proof}
		If 
		
		\begin{align*}\label{F2}
		\begin{array}{r}
		y=0\\
		b\frac{s+3p+4}{2} yx^p-c x^r-b^2\frac{r+1}{2} x^{2p+1}=0
		\end{array}	
		\end{align*}	
		
		then $\deg(Q)=max\{r,2p+1\}$, that is, we should consider two cases.
		\begin{itemize}
			\item If $\deg(Q)=2p+1$ then $x^r(c+b^2(\frac{r+1}{2})x^{2p+1-r})=0.$ This implies that $$x=0 \quad or \quad x^{\gamma}=\frac{-2c}{b^2(r+1)}$$
			
			being $\gamma=2p-r+1$. If $\gamma$ is even then it is necessary that $c<0$, and therefore the critical points are $(0,0)$, $\left(\sqrt[\gamma]{\frac{-2c}{b^2(r+1)}},0\right)$ and $\left(-\sqrt[\gamma]{\frac{-2c}{b^2(r+1)}},0\right)$.\\
			If $\gamma$ is odd then the critical points are $(0,0)$ and $\left(\sqrt[\gamma]{\frac{-2c}{b^2(r+1)}},0\right)$.
			
			\item  If $\deg(Q)=r$ analogously $x^{2p+1}(cx^{r-2p-1}+\frac{b^2(r+1)}{2})=0$. If $r-2p-1=\gamma$ is even then the critical points for the system \eqref{F2} are $(0,0)$, $\left(\sqrt[\gamma]{\frac{-2c}{b^2(r+1)}},0\right)$. On the other hand, if $\gamma$ is odd, it is necessary that $c<0$ and for instance the critical points are  $(0,0)$, $\left(\sqrt[\gamma]{\frac{-2c}{b^2(r+1)}},0\right)$ and $\left(-\sqrt[\gamma]{\frac{-2c}{b^2(r+1)}},0\right)$.\\  
		\end{itemize}

		Now for the family \eqref{F5}, we have that 
		\begin{align*}
		\begin{array}{r}
		y=0\\
		a\frac{3s+p+4}{2} yx^s-a\frac{r+1}{2} x^{2s+1}-c x^r=0.
		\end{array}	
		\end{align*}
		
		Then $\deg(Q)=max\{r,2s+1\}$, again we have to consider two cases.
		\begin{itemize}
			\item If $\deg(Q)=2s+1$ then $x^r\left(a\left(\frac{r+1}{2}\right)x^{2s+1-r}+c\right)=0.$ This implies that $$x=0 \quad or \quad x^{\gamma}=\frac{-2c}{a(r+1)}, \quad \gamma=2s-r+1.$$ If $\gamma$ is even then it is necessary that $c<0$, and for instance the critical points are $(0,0)$, $\left(\sqrt[\gamma]{\frac{-2c}{a(r+1)}},0\right)$ and $\left(-\sqrt[\gamma]{\frac{-2c}{a(r+1)}},0\right)$.\\
			If $\gamma$ is odd then the critical points are $(0,0)$ and $\left(\sqrt[\gamma]{\frac{-2c}{a(r+1)}},0\right)$.
			
			\item  If $\deg(Q)=r$ analogously $x^{2s+1}\left(cx^{r-2s-1}+\frac{a(r+1)}{2}\right)=0$. If $r-2s-1=\gamma$ is even then the critical points for the system \eqref{F5} are $(0,0)$, $\left(\sqrt[\gamma]{\frac{-2c}{a(r+1)}},0\right)$. On the other hand, if $\gamma$ is odd, it is necessary that $c<0$ and for instance the critical points are  $(0,0)$, $\left(\sqrt[\gamma]{\frac{-2c}{a(r+1)}},0\right)$ and $\left(-\sqrt[\gamma]{\frac{-2c}{a(r+1)}},0\right)$.
		\end{itemize} 
	\end{proof}
	
	\begin{prop}
		For systems of the form \eqref{F3}, \eqref{F4}, \eqref{F6} and \eqref{F7}, the point $(0,0)$ is the only critical point.   
	\end{prop}
	
	\begin{proof}
		We can see that the common characteristic in these families is that $c=0$. Then for the family \eqref{F3}
		
		\begin{align*}
		\begin{array}{r}
		y=0\bigskip\\
		\left(a\frac{3s+p+4}{2} x^s+b\frac{s+3p+4}{2} x^p \right)y-a\frac{s+p+2}{2} x^{2s+1}-b^2\frac{s+p+2}{2} x^{2p+1}=0
		\end{array}	
		\end{align*}
		
		Now we have two cases. If $\deg(Q)=2s+1$ then $x^{2p+1}\left(\frac{s+p+2}{2}\right)(x^{2s-2p}+b^2)=0$, where $s,p\in \mathbb{Z}^+$ and $2(s-p)$ is even. Therefore, we can conclude that the only solution for the systems under the conditions given above is $(0,0)$. 
		It follows analogously when $\deg(Q)=2p+1$.\\       
		
		For the family \eqref{F4}
		\begin{align*}
		\begin{array}{r}
		y=0\\
		-c x^{s+p+1}=0
		\end{array}	
		\end{align*}
		
		Then we can see that $(0,0)$ is the only critical point. For the family \eqref{F6}
		\begin{align*}
		\begin{array}{r}
		y=0\\
		b(m+p+1) yx^p-b^2m x^{2p+1}=0
		\end{array}	
		\end{align*}
		
		Then $b^2m x^{2p+1}=0$. Again $(0,0)$ is the only critical point. For the family \eqref{F7}
		\begin{align*}
		\begin{array}{r}
		y=0\\
		a(m+s+1)yx^s-am x^{2s+1}=0
		\end{array}	
		\end{align*}
		then $am x^{2s+1}=0$, therefore $(0,0)$ is the only critical point.
			\end{proof}
	
	
	\section{A biparametric quadratic polynomial vector field.}
In this section we consider the Polyanin-Zaitsev vector field  \eqref{din1}restricted to $a=0$, $m=\frac32$, $k=\frac12$, which also corresponds to the subcase given in Eq. \eqref{F2} being $a=0$, $r=2$, $m=\frac32$, $p=0$ and $k=1$. That is, we present the algebraic and qualitative study of the following planar polynomial differential system:
	\begin{equation}\label{eqexa}
	\begin{array}{ccl}
	\dot{x}&=&y \bigskip \\
	\dot{y}&=& \frac{5}{2}by-\frac{3}{2}b^2x-cx^2.
	\end{array}
	\end{equation}
The Riccatti equation associated to the planar differential system \eqref{eqexa} becomes to the first order differential equation given by

\begin{equation}\label{ricexa}
\left(-\frac{3}{2}t^2-c\right)\frac{dz}{dt}=-\frac{1}{2}b+\frac12tz,
\end{equation}
which is always integrable in the sense of differential Galois theory. Such solution is given by
$$z(t)=\frac{const-\int\frac{b}{(3t^2+2c)^{5/6}}dt}{(3t^2+2c)^{1/6}}$$
	The critical points of this systems are: $(0,0)$ and $(\frac{-3b^2}{2c},0)$.
	Now we study the behavior of the orbits near them.\\
	
	For the vector field $X(x,y)=(y,\frac{5}{2}by-\frac{3}{2}b^2x-cx^2)$ the linear part is:    
	\begin{displaymath}
	DX=\left(\begin{array}{cc}
	0& 1\\
	\frac{-3}{2}b^2-2cx& \frac{-5}{2}b
	\end{array}\right).
	\end{displaymath}
	
	For the origin 
	\begin{displaymath}
	DX(0,0)=\left(\begin{array}{cc}
	0& 1\\
	\frac{-3}{2}b^2& \frac{-5}{2}b
	\end{array}\right)
	\end{displaymath}
	and his characteristic polynomials is: $\lambda^2-\frac{5}{2}b\lambda+\frac{3}{2}b^2=0$, then the eigenvalues are $\lambda_1=\frac{3}{2}b$ and $\lambda_2=b.$ It is the origin is a stable (unstable) node if $b<0$ $(b>0)$.\\ 
	
	Now for the other critical point:
	\begin{displaymath}
	DX(\frac{-3b^2}{2c},0)=\left(\begin{array}{cc}
	0& 1\\
	\frac{3}{2}b^2& \frac{5}{2}b
	\end{array}\right)
	\end{displaymath}
	and his characteristic polynomials is: $\lambda^2-\frac{5}{2}b\lambda-\frac{3}{2}b^2=0$, then the eigenvalues are $\lambda_1=3b$ and $\lambda_2=\frac{-1}{2}b$. It is, the system has a saddle at the point.\\
	
	Now we will analysis the infinity behavior, using the Poincar\'e compactification. For this we will use the equivalent systems at the the chart $U_1$ and $U_2$ give by the variable change
	$$(x,y)=(\frac{1}{v}, \frac{u}{v}) \hspace*{1cm} (x,y)=(\frac{u}{v}, \frac{1}{v}),$$
	  
	respectively. For more see \cite{DllA}.
	
	At $U_1$ chart the system is: 
	
	\begin{displaymath}
	\begin{array}{ccl}
	\dot{u}&=&-u^2v+\frac{5}{2}buv-\frac{3}{2}b^2v-c \bigskip \\
	\dot{v}&=& -uv^2.
	\end{array}
	\end{displaymath}  
	
	With critical point $(0,-\frac{2c}{3b^2})$, but this is not on the equator of the Poincar\'e sphere.\\

	At $U_2$ chart the system is: 
	
	\begin{displaymath}
	\begin{array}{ccl}
	\dot{u}&=&v-\frac{5}{2}buv+\frac{3}{2}b^2u^2v+cu^3\bigskip \\
	\dot{v}&=& -\frac{5}{2}bv^2+\frac{3}{2}b^2uv^2+cu^2v.
	\end{array}
	\end{displaymath}
	
	Whit nilpotent singular point at the origin. Using theorem 3.5 (\cite[3.5]{DllA}), take account that $f(u)=-cu^3+...+O(u^6)$; $B(u,f)=-c^2u^5-\frac{5}{2}bc^2u^6+O(u^6)+...$; $G(u)=4cu^2+...+TOS$; $m=5$ odd; $a=-c^2$; $n=2$ even; $m=2n+1$; $b^2+4a(n+1)=4c^2>0$. Then $(0,0)$ is a repelling (attracting) node if $c>0$ ($c<0$). Next, we can see the global phase portrait in each case, which are given in figure \ref{fig:uno}, figure \ref{fig:dos}, figure \ref{fig:tres} and figure \ref{fig:cuatro}.

\begin{figure}[h!]
	\centering
	\begin{minipage}{0.47\textwidth}
		\centering
		\includegraphics[width=1.1\textwidth]{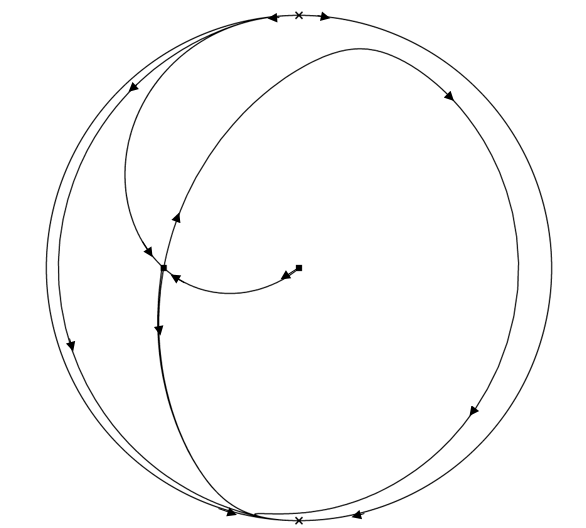}
		\caption{$b>0$; $c>0.$}
		\label{fig:uno}
	\end{minipage}%
	\hspace{5mm}
	\begin{minipage}{0.47\textwidth}
		\centering
		\includegraphics[width=1\textwidth]{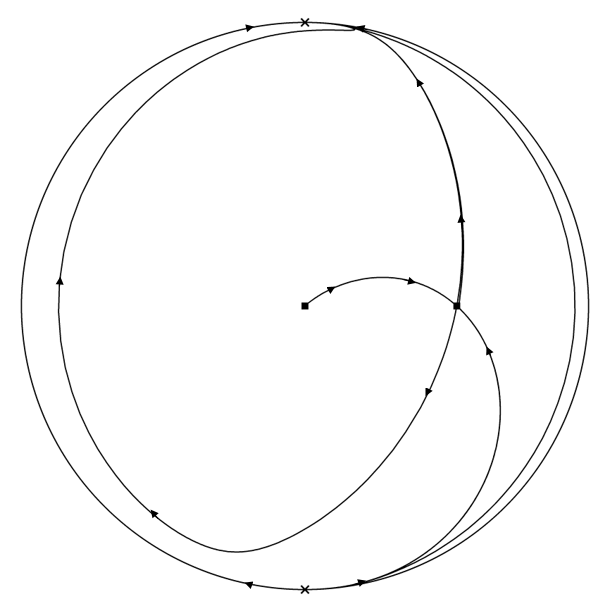}
		\caption{$b>0$; $c<0.$}
		\label{fig:dos}
	\end{minipage}
\end{figure}

\begin{figure}[h!]
	\centering
	\begin{minipage}{0.47\textwidth}
		\centering
		\includegraphics[width=1\textwidth]{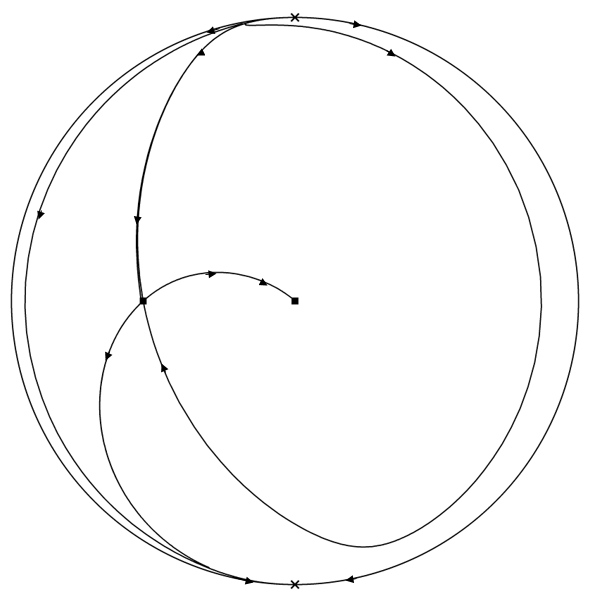}
		\caption{$b<0$; $c>0.$}
		\label{fig:tres}
	\end{minipage}%
	\hspace{5mm}
	\begin{minipage}{0.47\textwidth}
		\centering
		\includegraphics[width=1\textwidth]{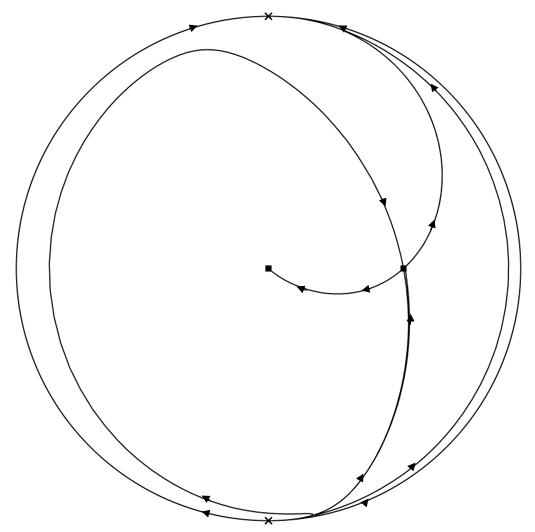}
		\caption{$b<0$; $c<0.$}
		\label{fig:cuatro}
	\end{minipage}
\end{figure}

	\section{Final Remarks}
	In this paper we studied from algebraic point of view  and singular points study the five parametric family of linear differential systems that came from the corrigendum of Exercise 11 in \cite[\S 1.3.3]{PY}, which we called  Polyanin-Zaitsev vector field. We solved the corrected exercise through a series of transformations using Hamiltonian changes of variables. An analysis was also developed to find critical points and their behavior.We included an example which corresponds to a biparametric quadratic Polyanin-Zaitsev vector field.\\

Although the title of this paper says qualitative remarks, we do not included the analysis of infinite critical points and the study of orbits. To obtain infinite critical points for Polyanin-Zaitsev vector field is not an easy task which we consider an open problem. The interested reader could try to solve and to complete the results of this paper for the vector field studied in this paper.



\begin{thebibliography}{99}
		%
		\bibitem{Ac3} P.B. Acosta-Hum\'anez, {\it Galoisian Approach to Supersymmetric Quantum Mechanics}. PHD. Thesis barcelona (2009). Available at \href{https://arxiv.org/abs/0906.3532}{arXiv:0906.3532}
		
		\bibitem{A} P.B. Acosta-Hum\'anez, {\it La Teor\'ia de Morales-Ramis y el Algoritmo de Kovacic} Lecturas Matem\'aticas Volumen Especial (2006) 21--56.
		\bibitem{AB} P.B. Acosta-Hum\'anez and D. Bl\'azquez-Sanz, \textit{Non-integrability of some hamiltonians with rational potentials}. Discrete and Continuous Dynamical Systems - Series B (DCDS-B) \textbf{10} (2008), 265--293. Available at \href{https://arxiv.org/abs/0610010}{arXiv:0610010}
		
		\bibitem{Almp}P.B. Acosta-Hum\'anez, J.T L\'azaro, J.J. Morales-Ruiz and Ch. Pantazi, {\it On the integrability of polynomial fields in the plane by means of Picard-Vessiot theory}. Discrete and Continuous Dynamical Systems \textbf{35} (2015), 1767--1800. Available at \href{https://arxiv.org/abs/1012.4796}{arXiv:1012.4796}.
		
		\bibitem{Amw} P. Acosta-Hum\'anez, J. Morales-Ruiz and J.-A. Weil, {\it Galoisian Approach to integrability of Schr\"odinger Equation}. Reports on Mathematical Physics, 67 (2011) 305--374.  Available at \href{https://arxiv.org/abs/1008.3445}{arXiv:1008.3445}. 
		\bibitem{Ap} P.B. Acosta-Hum\'anez, Ch. Pantazi, {\it Darboux Integrals for Schrodinger Planar Vector Fields via Darboux Transformations} SIGMA, 8 (2012) 043.  Available at \href{https://arxiv.org/abs/1111.0120}{arXiv:1111.0120}. 
		
		\bibitem{AJp} P.B. Acosta-Hum\'anez and J. Perez, {\it Teor\'ia de Galois diferencial: una aproximaci\'on} Matem\'aticas: Ense\~anza Universitaria \textbf{15} (2007), 91--102.
		
		
		\bibitem{AJp2} P.B. Acosta-Hum\'anez and J. Perez, {\it Una introducci\'on teor\'ia de Galois diferencial} Bolet\'in de Matem\'aticas Nueva Serie, 11 (2004) 138-149.
		
		
		
		
		
		\bibitem{Jg} J. Guckenheimer, {\it Nonlinear Oscillations, Dynamical Systems and Bifurcations of Vector Fields}. Springer-Verlag  New York (1983).
		
		
		\bibitem{GHW} Guckenheimer, J., Hoffman, K., and Weckesser, W., {\it The forced Van der Pol equation I: The slow flow and its bifurcations}, SIAM J. Applied Dynamical Systems, 2003, 2, 1-35.
		
		\bibitem{Tk} Kapitaniak, T., {\it Chaos for Engineers: Theory, Applications and Control}, Springer, Berlin, Germany, 1998.
		
		
		\bibitem{JJM} J. Morales-Ruiz, {\it Differential Galois Theory and Non-Integrability of Hamiltonian Systems}. Birkh\"auser, Basel (1999).
		
		\bibitem{NAY} Nagumo, J., Arimoto, S. and Yoshizawa, S. {\it An active pulse transmission line simulating nerve axon}, Proc. IRE, 1962, 50, 2061-2070.
		
		\bibitem{Ns} V.V. Nemytskii and V.V. Stepanov, {\it Qualitative Theory of Differential Equations}, Princeton University Press, Princeton (1960).
		
		
		
		\bibitem{PK} L. Perko, {\it Differential equations and Dynamical systems, Third Edition}. Springer-Verlag  New York, Inc (2001).
		
		\bibitem{PY} A.D. Polyanin and V.F. Zaitsev, {\it Handbook of exact solutions for ordinary differential equations, Second Edition}. Chapman and Hall, Boca Raton (2003).
		
		\bibitem{VDP} Van der Pol, B., and Van der Mark, J., {\it Frequency demultiplication}, Nature, 1927, 120, 363-364.
		
		
		\bibitem{VDPS} M. van der Put and M. Singer, {\it Galois Theory in Linear Differential Equations} . Springer-Verlag  New York (2003).
		
		\bibitem{JAW} J.A. Weil, {\it Constant et polyn\'omes de Darboux en alg\`ebre diff\'erentielle: applications aux syst\`emes diff\'erentiels lin\'eaires}. Doctoral thesis, (1995).
		
		\bibitem{DllA} F. Dumortier, J. Llibre, J.C. Art\'es, {\it Qualitative theory of planar differential systems}. Springer-Verlag Berlin Heidelberg. (2006).
		
		
		
		
		
		
		
		
	\end{thebibliography}
\end{document}